\documentclass[reqno,centertags, 12pt]{amsart}
\usepackage{amsmath,amsthm,amscd,amssymb}
\usepackage{esint}  
\usepackage{latexsym}
\usepackage{graphicx}
\usepackage{enumitem}
\usepackage[mathscr]{eucal}
\usepackage{float}

\newcommand{\bbC}{{\mathbb{C}}}

\newcommand{\bbR}{{\mathbb{R}}}

\newcommand{\bbZ}{{\mathbb{Z}}}

\newcommand{\calF}{{\mathcal F}}
\newcommand{\calG}{{\mathcal G}}

\newcommand{\calS}{{\mathcal S}}
\newcommand{\calT}{{\mathcal T}}

\newcommand{\lb}{\label}

\newcommand{\spec}{\text{\rm{spec}}}

\newcommand{\bi}{\bibitem}

\newcommand{\beq}{\begin{equation}}
\newcommand{\eeq}{\end{equation}}
\newcommand{\ba}{\begin{align}}
\newcommand{\ea}{\end{align}}

\newcounter{smalllist}

\newcommand{\comm}[1]{}

\allowdisplaybreaks
\numberwithin{equation}{section}

\newtheorem{theorem}{Theorem}[section]

\newtheorem{lemma}[theorem]{Lemma}
\newtheorem{corollary}[theorem]{Corollary}
\theoremstyle{definition}
\newtheorem{example}[theorem]{Example}
\newtheorem{conjecture}[theorem]{Conjecture}

\newtheorem*{remark}{Remark}
\newtheorem*{remarks}{Remarks}

\newcommand{\jap}[1]{\langle #1 \rangle}
\newcommand{\norm}[1]{\lVert#1\rVert}

\begin{document}

\title[Remarks on Trees]{Remarks on Periodic Jacobi Matrices on Trees}
\author[J.~S.~Christiansen, B.~Simon and M.~Zinchenko]{Jacob S.~Christiansen$^{1,4}$, Barry Simon$^{2,5}$ \\and
Maxim~Zinchenko$^{3,6}$}

\thanks{$^1$ Centre for Mathematical Sciences, Lund University, Box 118, 22100 Lund, Sweden.
 E-mail: stordal@maths.lth.se}

\thanks{$^2$ Departments of Mathematics and Physics, Mathematics 253-37, California Institute of Technology, Pasadena, CA 91125.
E-mail: bsimon@caltech.edu}

\thanks{$^3$ Department of Mathematics and Statistics, University of New Mexico,
Albuquerque, NM 87131, USA; E-mail: maxim@math.unm.edu}

\thanks{$^4$ Research supported by the Swedish Research Council (VR) Grant No. 2018-03500.}

\thanks{$^5$ Research supported in part by NSF grant DMS-1665526.}

\thanks{$^6$ Research supported in part by Simons Foundation grant CGM-581256.}

\

\date{\today}

\begin{abstract}  We look at periodic Jacobi matrices on trees.  We provide upper and lower bounds on the gap of such operators analogous to the well known gap in the spectrum of the Laplacian on the upper half-plane with hyperbolic metric.  We make some conjectures about antibound states and make an interesting observation for what \cite{ABS} calls the $rg$-model.
\end{abstract}

\maketitle

\section{Introduction} \lb{s1}

This paper is a contribution to a growing recent literature on periodic Jacobi matrices on trees.  Specifically, we were motivated by the work of Avni, Breuer and Simon \cite{ABS} (see also \cite{ABKS}) whose notation and ideas we will follow; in particular, we refer the reader to that paper for the definitions from graph theory that we will use.  We note also the relevance of some recent preprints by a group at Berkeley \cite{GVK, BGVM}.

One starts out with a finite leafless graph, $\calG$, and a Jacobi matrix on that graph.  By this we mean a matrix with indices labelled by the vertices in the graph, whose diagonal elements are a function, $b$, on the vertices and off diagonal elements, which are only non-zero for pairs of vertices which are the two ends of some edge, with matrix elements determined by a function, $a$, on the edges.  The universal cover, $\calT$, of $\calG$ is always a leafless tree.  There is a unique and natural lift of any Jacobi matrix, $J$, on $\calG$ to an operator, $H$, (still called a Jacobi matrix) on $\ell^2(\calG)$.  Because $H$ is invariant under a group of deck transformations on $\calT$ (see \cite[Section 1.7]{CAA} for the covering space language we exploit), we call $H$ a \emph{periodic Jacobi matrix}.

There are three big general theorems in the subject: (1) a result of Sunada \cite{Sun} on labelling of gaps in the spectrum that implies the spectrum of $H$ is at most $p$ bands where $p$ is the number of vertices in the underlying finite graph $\calG$, (2) a result of Aomoto \cite{AomotoPoint} stating that if $G$ has a fixed degree then $H$ has no point spectrum, and (3) a result of Avni--Breuer--Simon \cite{ABS} that there is no singular spectrum because matrix elements of the resolvent are algebraic functions.  Besides a very few additional theorems, the subject at this point is mainly some interesting examples and lots of conjectures and questions.

This paper has some additional results, observations and an interesting fact about one fascinating example.

Typically, $\sigma$, the largest eigenvalue of $J$, does not lie in the spectrum of $H$.  Two of the subjects we treat involve $\sigma$.  Section \ref{s2} discusses quantitative estimates on the distance from $\sigma$ to the top of the spectrum of $H$.  Section \ref{s3} will show that in a number of examples, $\sigma$ can be viewed as an antibound state and we make a general conjecture that it is a pole of the analytically continued Green's function.  Section \ref{s4} studies an example of a finite graph with non-fixed degree for which Aomoto \cite{AomotoPoint} showed that $H$ has a point eigenvalue and ABS \cite{ABS} found an explicit eigenfunction.  We prove that this eigenfunction and its translates span the eigenspace.

\section*{Acknowledgment}

Jean Bourgain made many deep contributions to the spectral theory of Jacobi matrices and we dedicate this paper to his memory. JSC and MZ would like to thank F. Harrison and E. Mantovan for the hospitality of Caltech where some of this work was done.

\section{Gap Comparison Theorem} \lb{s2}

In this section we want to fix a connected graph, $\calG$, with $p$ vertices, $V(\calG)$, and $q$ edges, $E(\calG)$.  We are going to want to fix two sets of Jacobi parameters on $\calG$, namely $\{a_\alpha,b_j\}_{\alpha\in E(\calG), j\in V(\calG)}$ and $\{\tilde{a}_\alpha,\tilde{b}_j\}_{\alpha\in E(\calG), j\in V(\calG)}$, with corresponding Jacobi matrices $J$ and $\tilde{J}$ on $\ell^2(\calG)$ and to look at the induced periodic Jacobi matrices $H$ and $\tilde{H}$ on $\ell^2(\calT)$ with $\calT$ the universal cover of $\calG$.

By the Perron--Frobenius theorem \cite[Theorem 1.4.4]{BaRa97}, $J$ has a simple largest eigenvalue, $\sigma(a,b)$, with unique strictly positive normalized eigenvector, $\psi(a,b,j)_{j\in V}$. (Below in the bipartite case, where we use $\sigma_-$, we will sometimes write $\sigma_+$ for $\sigma$.) We will use $\Sigma(a,b)$ for the $\sup$ of the spectrum of $H$.
It is a fundamental result of Sy--Sunada \cite{SS} that if the graph $\calG$ has more than one independent closed loop (so the fundamental group is non-amenable), then the \emph{gap},
\begin{equation}\label{2.1}
  G(a,b)\equiv \sigma(a,b) - \Sigma(a,b)
\end{equation}
is strictly positive.  In this section, our goal is to prove comparison inequalities for the gaps $G(a,b)$ and $G(\tilde{a},\tilde{b})$ based on comparison results in distinct but related contexts by Kirsch--Simon \cite{KiSi} and Frank--Simon--Weidl \cite{FSW}.

These results rely on a ground state representation for the quadratic form of $J$, a formulae that for continuum Sturm--Liouville operators goes back to Jacobi \cite{Jac}.  For operators, $H\equiv -\Delta+V$, on $L^2(\bbR^\nu)$ for which there is a strictly positive solution of the PDE, $(-\Delta+V)\psi=E_0\psi$, it takes the form
\begin{equation}\label{2.2}
  \jap{f\psi,(H-E_0)f\psi} = \int |\nabla f(x)|^2 \psi(x)^2 \,d^\nu x
\end{equation}
This was first used in spectral theory by Birman \cite{Birman} and made popular in constructive Hamiltonian quantum field theory after its use by Nelson \cite{Nelson} and then Glimm, Jaffe, Sigal and others \cite{GJQFT, SiPphi2}.  Remarkably, the version for Jacobi-type matrices seems to have only been written down in 2008 by Frank--Simon--Weidl \cite{FSW} (see Keller et al. \cite{KPP}).  For our situation, it takes the following form:

\begin{theorem} \lb{T2.1} Let $J$ be a Jacobi matrix on a finite graph, $\calG$, or a periodic Jacobi matrix on an infinite tree, $\calT$.
Suppose $\psi$ is a nonnegative sequence $\{\psi_i\}_{i\in V}$ that obeys $J\psi=\sigma\psi$.
Then for any real sequence $\{f_i\}_{i\in V}$ with $f\psi\in\ell^2(V)$, we have that

\begin{equation}\label{2.3}
  \jap{f\psi,(\sigma-J)f\psi} = \sum_{\alpha=(j,k)} a_\alpha \psi_j\psi_k(f_j-f_k)^2
\end{equation}
\end{theorem}

\begin{remarks} 1. The sum in \eqref{2.3} is over all edges in the graph or tree and the notation indicates that $j$ and $k$ are the two ends of the edge.

2. For the finite graph, $\psi$ is just the eigenfunction for the largest eigenvalue of the graph.  For the tree, $\psi$ is the (periodic) lift of the underlying graph $\psi$ to its universal cover.

3. We use $V$ as shorthand notation for either $V(\calG)$ or $V(\calT)$.

4.  The proof is as in \cite[Theorem 3.2]{FSW}.  If $M_f$ is multiplication by a function $f$ of finite support, one verifies that
\begin{equation} \label{2.3a}
  [M_f,(J(a,b)-\sigma)] = J(a_\alpha(f_j-f_k),b\equiv 0)
\end{equation}
which then implies
\begin{equation} \label{2.3b}
  [M_f,[M_f,(J(a,b)-\sigma)]] = J(a_\alpha(f_j-f_k)^2,b\equiv 0)
\end{equation}
Letting $\psi^\sharp$ be $\psi$ restricted to the support of $f$ and all vertices connected to this finite support then yields $M_f(J(a,b)-\sigma)\psi^\sharp=0$ and hence
\begin{equation} \label{2.3c}
  \jap{\psi^\sharp,[M_f,[M_f,(J(a,b)-\sigma)]]\psi^\sharp} = -2\jap{f\psi,(J(a,b)-\sigma)f\psi}
\end{equation}
so \eqref{2.3} for finitely supported $f$ follows from \eqref{2.3b} and \eqref{2.3c}.
To get \eqref{2.3} for general $f$ with $f\psi\in\ell^2(V)$, one approximates $f$ by finitely supported $f_n$ of increasing support such that 
$\|(f-f_n)\psi\|_{\ell^2(V)}\to 0$ and takes limits of both sides. 
The LHS converges since $J(a,b)$ is a bounded operator and the RHS converges since we can view it as the difference of
\begin{equation}
S_1(f_n)=\sum_{\alpha=(j,k)} a_\alpha \psi_j\psi_k \bigl[(f_n)_j^2+(f_n)_k^2\bigr]
\end{equation}
and 
\begin{equation}
S_2(f_n)=2\sum_{\alpha=(j,k)} a_\alpha (f_n\psi)_j(f_n\psi)_k = 2\jap{f_n\psi,J(a,0)f_n\psi}
\end{equation}
where $S_1(f_n)\to S_1(f)$ by monotone convergence and $S_2(f_n)\to S_2(f)$ since $J(a,0)$ is a bounded operator.

\end{remarks}

We care about \eqref{2.3} because with $\norm{f}_\psi=\norm{f\psi}_{\ell^2(V(\calT))}$ and $\calF$ the sequences of finite support, we have that
\begin{equation}\label{2.4}
  G(a,b) = \inf_{f\in\calF} \left[ \left(\norm{f}_\psi\right)^{-2} \sum_{\alpha=(j,k)} a_\alpha \psi_j\psi_k(f_j-f_k)^2 \right]
\end{equation}
This variational principle immediately implies our basic comparison theorem:

\begin{theorem} \lb{T2.2} Fix a finite graph, $\calG$, and two sets of Jacobi parameters, $\{a_\alpha,b_j\}_{\alpha\in E(\calG), j\in V(\calG)}$ and $\{\tilde{a}_\alpha,\tilde{b}_j\}_{\alpha\in E(\calG), j\in V(\calG)}$.  Let $\psi$ and $\tilde\psi$ be the positive solutions of $J\psi=\sigma\psi$ and $\tilde J\tilde\psi=\tilde\sigma\tilde\psi$, respectively, and set
\begin{equation}\label{2.5}
  S = \sup_{\alpha=(i, j)\in E(\calG)} \left[\frac{a_\alpha}{\tilde{a}_\alpha}\frac{\psi_i\psi_j}{\tilde{\psi}_i\tilde{\psi}_j}\right], \qquad \tilde{S}= \sup_{\alpha=(i, j)\in E(\calG)} \left[\frac{\tilde{a}_\alpha}{a_\alpha}\frac{\tilde{\psi}_i\tilde{\psi}_j}{\psi_i\psi_j}\right]
\end{equation}
\begin{equation}\label{2.6}
  I = \inf_{i\in V(\calG)} \frac{\psi_i^2}{\tilde{\psi}_j^2}, \qquad \tilde{I}= \inf_{i\in V(\calG)} \frac{\tilde{\psi}_j^2}{\psi_i^2}
\end{equation}
Then
\begin{equation}\label{2.7}
 \frac{\tilde{I}}{\tilde{S}} G(\tilde{a},\tilde{b}) \le  G(a,b) \le \frac{S}{I} G(\tilde{a},\tilde{b})
\end{equation}
\end{theorem}

\begin{remark} It is easy to see that one need not take a normalized $\psi$ or $\tilde{\psi}$ in these definitions since normalization constants drop out of the ratios of $S/I$.
\end{remark}

We can also define a lower gap by letting $\sigma_-(a,b)$ to be the lowest eigenvalue of $J$, $\Sigma_-(a,b)$ the $\inf$ of the spectrum of $H$ and
\begin{equation}\label{2.8}
  G_-(a,b)\equiv \Sigma_-(a,b) - \sigma_-(a,b)
\end{equation}

A graph $\calG$ is called \emph{bipartite} if $V(G)$ can be written as a disjoint union of $V_1$ and $V_2$ so that every edge $\alpha\in E(\calG)$ has one end in $V_1$ and one end in $V_2$.  In that case, if $U$ is the unitary operator that multiplies components of vector, $u_i$, with $i\in V_1$ (resp. $i\in V_2$) by $1$ (resp. $-1$), then
\begin{equation}\label{2.9}
  U(-J(a,b))U^{-1}=J(a,-b)
\end{equation}
which implies that

\begin{equation}\label{2.10}
  \sigma_-(a,b) = - \sigma(a,-b), \quad \Sigma_-(a,b) = - \Sigma(a,-b), \quad G_-(a,b) = G(a,-b)
\end{equation}

Moreover, for bipartite graphs, \eqref{2.9} implies there is a vector $\psi^{(-)}\in\ell^2(V(\calG))$ with $U\psi^{(-)}$ positive and so that $J(a,b)\psi^{(-)}=\sigma_-(a,b)\psi^{(-)}$.  Thus if we define $S^{(-)}$ and $I^{(-)}$ with $\psi$ replaced by $\psi^{(-)}$, we find that

\begin{theorem} \lb{T2.3} Fix a finite bipartite graph, $\calG$, and two sets of Jacobi parameters, $\{a_\alpha,b_j\}_{\alpha\in E(\calG), j\in V(\calG)}$ and $\{\tilde{a}_\alpha,\tilde{b}_j\}_{\alpha\in E(\calG), j\in V(\calG)}$.  Define $S^{(-)}$, $I^{(-)}$, $\tilde{S}^{(-)}$, and $\tilde{I}^{(-)}$ as above.  Then

\begin{equation}\label{2.11}
 \frac{\tilde{I}^{(-)}}{\tilde{S}^{(-)}} G_-(\tilde{a},\tilde{b}) \le  G_-(a,b) \le \frac{S^{(-)}}{I^{(-)}} G_-(\tilde{a},\tilde{b})
\end{equation}
\end{theorem}

These are especially interesting if we have one comparison operator where we can compute everything.  If $\calG$ is a finite degree $d$ graph so that $\calT$ is a degree $d$ homogeneous tree, we can take $\tilde{a}_\alpha=1$ for all $\alpha\in E(\calG)$ and $\tilde{b}_j=0$ for all $j\in V(\calG)$.  Then $\tilde{\psi}_j=1$ for all $j$ is an unnormalized positive eigenvector with $\tilde{\sigma}=d$.  By \cite[Example 7.1]{ABS}, we have that $\tilde{\Sigma}=\sqrt{4(d-1)}$ so $\tilde{G}=d-\sqrt{4(d-1)}$.
Thus we have what we regard as the main result of this section:

\begin{corollary} \lb{C2.4} Fix a finite graph, $\calG$, of constant degree $d$ and a set, $\{a_\alpha,b_j\}_{\alpha\in E(\calG), j\in V(\calG)}$, of Jacobi parameters. Let
\begin{equation}\label{2.12}
  S = \sup_{\alpha=(ij)\in E(\calG)} \left[a_\alpha\psi_i\psi_j\right], \qquad \tilde{S}= \sup_{\alpha=(ij)\in E(\calG)} \left[\frac{1}{a_\alpha\psi_i\psi_j}\right]
\end{equation}
\begin{equation}\label{2.13}
  I = \inf_{i\in V(\calG)}\psi_i^2, \qquad \tilde{I}= \inf_{i\in V(\calG)} \frac{1}{\psi_i^2}
\end{equation}
Then
\begin{equation}\label{2.14}
 \frac{\tilde{I}}{\tilde{S}}\left(d-\sqrt{4(d-1)}\right) \le  G(a,b) \le \frac{S}{I}\left(d-\sqrt{4(d-1)}\right)
\end{equation}
\end{corollary}

Not all finite graphs of homogeneous degree $d$ are bipartite (e.g., the complete graph with $d+1$ vertices. The lowest eigenvalue of $J=\{a\equiv1,b\equiv0\}$ on this complete graph is $\sigma_-(a,b) = -1$ so $\sigma_-(a,b)\neq-\sigma(a,-b)=-d$) but the universal cover, $\calT_d$, which is the homogeneous tree of degree $d$, is always bipartite (since every tree is bipartite).  If $\calG$ is a bipartite degree $d$ finite graph, it is easy to see that $\tilde{\sigma} = d$ (the constant function is an eigenvector) so $\tilde{G}_+ = d-\sqrt{4(d-1)}$.  Since $\calG$ is bipartite, $\tilde{G}_- = d-\sqrt{4(d-1)}$, and so there is a result similar to (2.14) for $G_-(a,b)$.


Another explicit example is the $rg$-model of \cite[Example 7.2]{ABS}, where $r>g$ are two positive integers.  The underlying graph, $\calG$, has $r$ red vertices and $g$ green vertices so $p=r+g$. There are $rg$ edges, one between each pair of differently colored vertices. The natural comparison Jacobi parameters have all $\tilde{a}=1$ and all $\tilde{b}=0$ which is the model studied in \cite[Example 7.2]{ABS} and in Section \ref{s4} below. If you look at a vector where all red vertices have value $u$ and all green value $v$ and ask it be an eigenvector of the Jacobi matrix $\tilde{J}$ with eigenvalue $\lambda$, then the equations are
\begin{equation}\label{2.15}
  ru=\lambda v, \qquad gv = \lambda u
\end{equation}
whose solutions are easily seen to be (unique up to normalization)
\begin{equation}\label{2.16}
  \lambda = \pm\sqrt{rg}, \qquad v=\sqrt{r}, \quad u=\pm\sqrt{g}
\end{equation}
It is easily seen that $\tilde{J}$ has rank $2$ so the orthogonal complement of these eigenvectors is $\ker(\tilde{J})$ and thus
\begin{equation}\label{2.17}
  \tilde{\sigma} = \sqrt{rg}, \quad \tilde{\sigma}_- = -\sqrt{rg}
\end{equation}
By \cite[(7.16)]{ABS}, we have
\begin{equation}\label{2.18}
  \tilde{\Sigma} = (\sqrt{r-1}+\sqrt{g-1})^2, \quad \tilde{\Sigma}_- = -(\sqrt{r-1}+\sqrt{g-1})^2
\end{equation}
which yields explicit formulae for $\tilde{G}$ and $\tilde{G}_-$ and explicit comparison formula.

We end this section by noting that these comparison results imply non-zero gaps for all the associated periodic tree Jacobi matrices without using Sy--Sunada and with explicit bounds.

\section{AntiBound Conjecture} \lb{s3}

As discussed above, there is a positive periodic eigensolution for any periodic Jacobi matrix on a tree but the corresponding energy, $\sigma$, is not in the spectrum if the tree isn't $\bbZ$.  The norm of the eigenfunction on a ball of radius $R$ is exponentially growing and it is only polynomial growth that implies connection to the spectrum.  In non-relativistic quantum scattering theory \cite{GW, New}, such purely exponentially growing solutions are called \emph{anti-bound states} and correspond to second sheet poles at real energies not in the spectrum. We think the same might be true in this case.  We start with some explicit examples for trees.

\begin{example} \lb{E3.1} The free Laplacian on a degree $d$ homogeneous tree has a Green's function (diagonal matrix element of the resolvent) that is well-known and computed, for example, in \cite[Example 7.1]{ABS}.  \cite[(7.3)]{ABS} says that
\begin{equation}\label{3.1}
  G(z) = \frac{(2-d)z+d\sqrt{z^2-4(d-1)}}{2(d^2-z^2)}
\end{equation}
Initially this function is defined on $\bbC\setminus\spec(H)$ where
\begin{equation}
\spec(H)=[-s_d,s_d]; \quad s_d=\sqrt{4(d-1)}
\end{equation}
is determined as the set where $G$ has a boundary value with non-zero imaginary part.  As noted in the last section, $\sigma=d$ is the largest eigenvalue of the underlying finite Jacobi matrix.  It appears that $G(z)$ has a pole at $z=\pm d$ but, in fact, since we take the branch of square root which is positive on $(s_d,\infty)$, that square root at $z=d$ is $(d-2)$ and the numerator vanishes so there is no pole at $z=\pm d$ where $G$ is originally defined. However, $G$ has a meromorphic continuation to a two-sheeted Riemann surface and on the second sheet, the square root takes the value $\mp(d-2)$ at $z=\pm d$, so $G$ has second sheet poles at $\pm\sigma$.
\end{example}

\begin{example} \label{E3.2} The $rg$-model mentioned at the end of the last section is studied as \cite[Example 7.2]{ABS}.  \cite[(7.13)/(7.14)]{ABS} compute the Green's functions $G_r, G_g$ at red and green sites
\begin{align}
  G_r(z) &= \frac{(2-g)z^2-g\left[(r-g)-\sqrt{\Phi(z)}\right]}{2z(rg-z^2)} \label{3.2} \\
  G_g(z) &= \frac{(2-r)z^2-r\left[(g-r)-\sqrt{\Phi(z)}\right]}{2z(rg-z^2)} \label{3.3} \\
  \Phi(z) &= z^4+2(2-(r+g))z^2+(r-g)^2 \label{3.4}
\end{align}
If $\eta_\pm=\sqrt{r-1}\pm\sqrt{g-1}$, then $\spec(H)=[-\eta_+,-\eta_-]\cup[\eta_-,\eta_+]\cup\{0\}$ and the Green's functions are initially defined on $\bbC\setminus\spec(H)$.  As we saw above, $\sigma_\pm=\pm\sqrt{rg}$.  The denominators of \eqref{3.2}/\eqref{3.3} vanish at $\sigma_\pm$ but a calculation shows that the numerators also vanish there on the original domain of definition so the only pole is at $z=0$ for $G_r$.  Again, there is a meromorphic continuation to a two-sheeted Riemann surface with poles at $\sigma_\pm$.
\end{example}

\begin{example} \lb{E3.3} The last example where \cite{ABS} computes the Green's functions is the Jacobi matrix over the following graph
\begin{figure}[H]
\includegraphics[scale=0.6,clip=true]{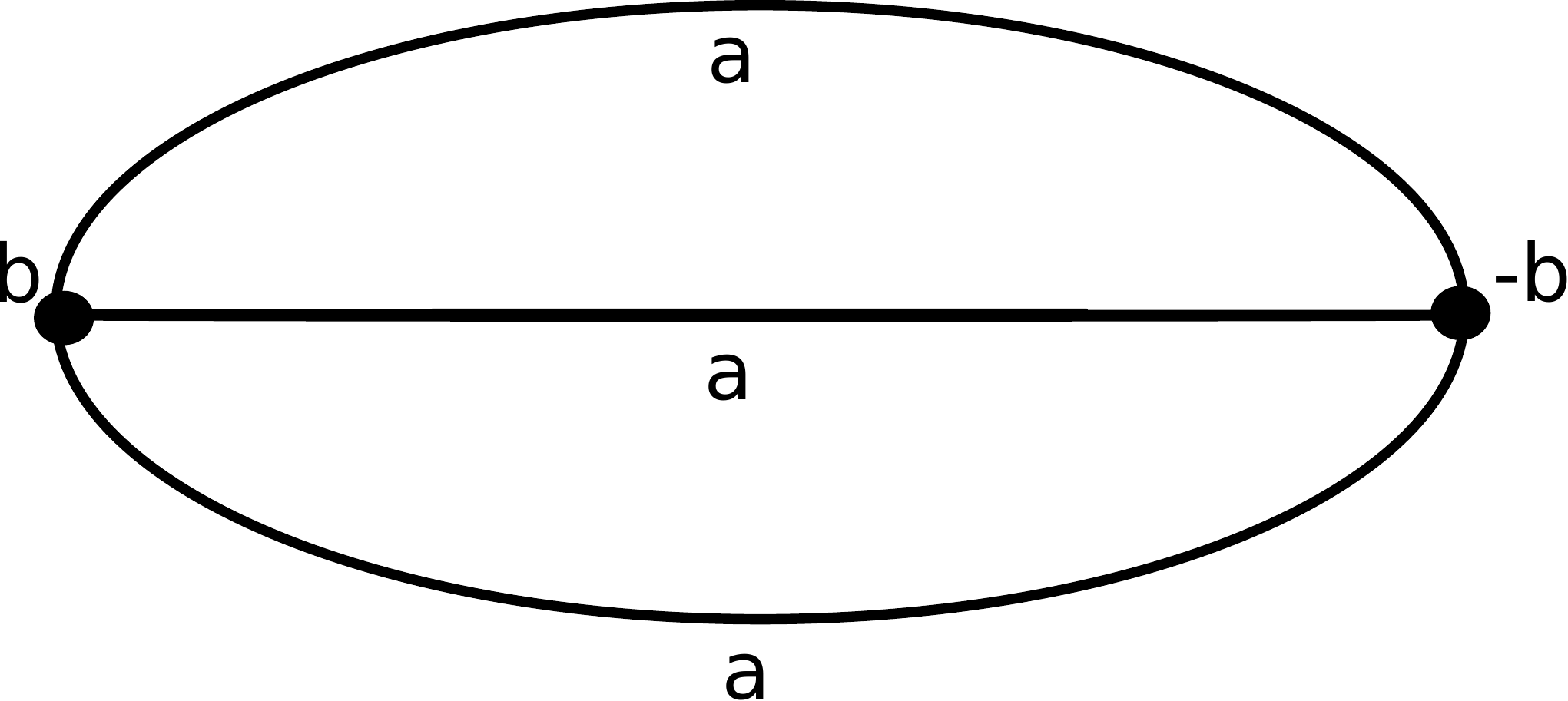}
\caption{Example \ref{E3.3}}
\end{figure}
\noindent whose covering tree is the homogeneous degree $3$ tree with all $a=1$ and alternate values $b$ and $-b$ for the potential.  The Green's functions are \cite[(7.25)]{ABS}
\begin{equation}\label{3.5}
  G_\pm(z)=\frac{(b^2-z^2)+3\sqrt{\Delta}}{2(z\mp b)(9-z^2+b^2)}; \quad \Delta=(z^2-b^2)^2-8(z^2-b^2)
\end{equation}
The spectrum of $H$ is 
\begin{equation}
\Bigl[-\sqrt{b^2+8},-b\Bigr]\cup\Bigl[b,\sqrt{b^2+8}\Bigr]
\end{equation}
$J$ is the $2\times 2$ matrix $\left(\begin{smallmatrix}
                                                                                           b & 3 \\
                                                                                           3 & -b
                                                                                         \end{smallmatrix}\right)$ which has eigenvalues
$\sigma_\pm=\pm\sqrt{b^2+9}$.  Since $\Delta(\sigma_\pm)=9$, we see that the numerator and denominator of \eqref{3.5} vanish at $\sigma_\pm$ so on the original domain of definition, $\bbC\setminus\spec(H)$, $G$ has no poles, but, once again, there are second sheet poles at $\sigma_\pm$.
\end{example}

Given all these examples, we state the following

\begin{conjecture} \lb{C3.4} For any periodic Jacobi matrix, $H$, on a tree (that is not $\bbZ$), the Green's functions analytically continued from $\bbC\setminus\spec(H)$ have poles on a higher sheet at $z=\sigma$ and, in the bipartite case, also at $\sigma_-$.
\end{conjecture}


Recall \cite{LPS} that certain finite graphs are called Ramanujan because $\sigma$ or $\pm\sigma$ are singled out from the other eigenvalues.  We wonder if there is a class of finite graph Jacobi matrices for which the only higher sheet poles of the Green's functions of its covering tree are at $\pm\sigma$.

\section{The $rg$-Model} \lb{s4}

The $rg$-model \cite[Example 7.2]{ABS} mentioned at the end of Section \ref{s2} was originally emphasized by Aomoto \cite{AomotoPoint} because he showed, by proving one of its Green's functions had a pole, it has a point eigenvalue at $0$ which is outside the continuous spectrum of $H$. Earlier, Godsil--Mohar \cite{GM88} had also noted the model had a spectral measure with a pure point at zero and computed the weight it contributes to the IDS. \cite{ABS} wrote down an explicit zero energy eigenvector, namely view the corresponding tree $\calT_{r,g}$ (with $r>g$) as a tree centered at a single red vertex at level $0$, with $g$ vertices at level $1$, then $g(r-1)$ vertices at level $2$, each level $1$ vertex linked to $r-1$ level $2$ vertices, etc. Thus, level $2k-1; k=1,2,\ldots$ has $g[(r-1)(g-1)]^{k-1}$ green vertices while level $2k; k=1,2,\ldots$ has $g(r-1)[(r-1)(g-1)]^{k-1}$ red vertices.  The ABS eigenfunction is radially symmetric (i.e., of constant value on each level) with
\begin{equation}\label{4.1}
\begin{aligned}
  &u(2k-1) = 0; \quad k=1,2,\ldots, \\
  &u(2k) = (-1/(r-1))^k; \quad k=0,1,\ldots
\end{aligned}
\end{equation}
This $u$ has $\ell^2$-norm
\begin{align}\label{4.2}
  \norm{u}_2^2 &= 1+\sum_{k=1}^{\infty} g(r-1)[(r-1)(g-1)]^{k-1} (-1/(r-1))^{2k} \nonumber\\
               &=1+ \frac{g}{r-1}\sum_{n=1}^\infty\Big(\frac{g-1}{r-1}\Big)^{n-1}
               = 1 + \frac{g}{r-g} = \frac{r}{r-g}
\end{align}
and it is easy to confirm that $Hu=0$.

This model has a symmetry group, $\calS$, that includes translations that map any red vertex into any other red vertex.  Our goal in this section is to prove

\begin{theorem} \lb{T4.1} $u$ and its images under the symmetry $\calS$ span the eigen\-space, $\ker(H)$, in that the null vector is the only vector in this space which is orthogonal to $u$ and all its translates.
\end{theorem}

Our proof relies on \eqref{4.2} and two formulae from \cite{ABS}, namely that $G_g$ is regular at $0$ and \cite[(7.17)]{ABS} that near $z=0$,
\begin{equation}\label{4.3}
  G_r(z)= -\frac{r-g}{rz} + \text{O}(1)
\end{equation}

\begin{remark} As noted in \cite[Example 8.5]{ABS}, the Aomoto Index Theorem implies that the DOS measure of the eigenvalue is $(r-g)/(r+g)$.  As also noted in \cite{ABS}, this follows from the regularity of $G_g$ at $0$ and \eqref{4.3}.  One can turn this around and show that \eqref{4.3} follows from Aomoto's calculation of the DOS weight of the eigenvalue.
\end{remark}

\begin{lemma} \lb{L4.2} The residue of $G_r$ at $z=0$ is $-|u(0)|^2/\norm{u}_2^2$.
\end{lemma}

\begin{proof} Immediate from \eqref{4.2}, \eqref{4.3} and $u(0)=1$.
\end{proof}

\begin{proof} [Proof of Theorem \ref{T4.1}] Let $e_1=u/\norm{u}_2$ and pick $\{e_j\}_{j=2}^\infty$ an orthonormal basis for the orthogonal complement of $e_1$ in $\ker(H)$. By the spectral theorem, the residue of the pole of $\jap{\delta_\ell,(H-z)^{-1}\delta_\ell}$ at $z=0$ is just $-\sum_{j=1}^{\infty}|e_j(\ell)|^2$.  Since $G_g$ has no pole at $z=0$, we conclude every $e_j$ vanishes at every green vertex which implies that every $\varphi\in\ker(H)$ vanishes at all green vertices.  By the above lemma, we conclude that for every $j=2,3,\ldots$, $e_j(0)$ vanishes.  It follows that if $\varphi\in\ker(H)$ is orthogonal to $u$, then $\varphi(0)=0$.

Now let $\varphi\in\ker(H)$ be orthogonal to $u$ and all its translates under $\calS$.  Since $\calS$ acts transitively on the red vertices, the argument in the last paragraph implies that $\varphi$ vanishes on the red vertices.  Since $\varphi\in\ker(H)$, it vanishes at the green vertices.  Thus it is the null vector as claimed. \end{proof}

As we were preparing this paper, which represents work mainly done in June 2019, we received an early draft of a very interesting paper of Banks et al. \cite{BGVM} that provides a lot of information about point spectrum of periodic Jacobi matrices on trees.  It seems to us possible that with the methods of \cite{BGVM} one can extend Theorem \ref{T4.1} to a much more general context.


\end{document}